\documentclass[11pt,reqno]{article}

\usepackage{subfigure}

\usepackage{latexsym}
\usepackage{graphics}
\usepackage{amsmath}
\usepackage{xspace}
\usepackage{amssymb}
\usepackage{psfrag}
\usepackage{epsfig}
\usepackage{amsthm}
\usepackage{pst-all}
\usepackage{hyperref}

\usepackage{color}

\definecolor{Red}{rgb}{1,0,0}
\definecolor{Blue}{rgb}{0,0,1}
\definecolor{Olive}{rgb}{0.41,0.55,0.13}
\definecolor{Yarok}{rgb}{0,0.5,0}
\definecolor{Green}{rgb}{0,1,0}
\definecolor{MGreen}{rgb}{0,0.8,0}
\definecolor{DGreen}{rgb}{0,0.55,0}
\definecolor{Yellow}{rgb}{1,1,0}
\definecolor{Cyan}{rgb}{0,1,1}
\definecolor{Magenta}{rgb}{1,0,1}
\definecolor{Orange}{rgb}{1,.5,0}
\definecolor{Violet}{rgb}{.5,0,.5}
\definecolor{Purple}{rgb}{.75,0,.25}
\definecolor{Brown}{rgb}{.75,.5,.25}
\definecolor{Grey}{rgb}{.5,.5,.5}

\newcommand{\pr}{\mathbb{P}}
\newcommand{\E}[1]{\mathbb{E}\!\left[#1\right]}
\newcommand{\R}{\mathbb{R}}

\newcommand{\N}{{\bf \mathcal{N}}}

\newcommand{\sign}{\mathrm{sign}}

\newcommand{\ignore}[1]{\relax}

\newtheorem{theorem}{Theorem}[section]

\newtheorem{lemma}[theorem]{Lemma}
\newtheorem{conjecture}[theorem]{Conjecture}
\newtheorem{prop}[theorem]{Proposition}

\newtheorem{coro}[theorem]{Corollary}

\newtheorem{Defi}[theorem]{Definition}

\newtheorem{Assumption}[theorem]{Assumption}

\newtheorem{algorithm}[theorem]{Algorithm}
\newtheorem*{lemma*}{Lemma}

\definecolor{Red}{rgb}{1,0,0}
\definecolor{Blue}{rgb}{0,0,1}
\definecolor{Olive}{rgb}{0.41,0.55,0.13}
\definecolor{Green}{rgb}{0,1,0}
\definecolor{MGreen}{rgb}{0,0.8,0}
\definecolor{DGreen}{rgb}{0,0.55,0}
\definecolor{Yellow}{rgb}{1,1,0}
\definecolor{Cyan}{rgb}{0,1,1}
\definecolor{Magenta}{rgb}{1,0,1}
\definecolor{Orange}{rgb}{1,.5,0}
\definecolor{Violet}{rgb}{.5,0,.5}
\definecolor{Purple}{rgb}{.75,0,.25}
\definecolor{Brown}{rgb}{.75,.5,.25}
\definecolor{Grey}{rgb}{.5,.5,.5}
\definecolor{Pink}{rgb}{1,0,1}
\definecolor{DBrown}{rgb}{.5,.34,.16}
\definecolor{Black}{rgb}{0,0,0}

\newcommand{\calE}{\mathcal{E}}
\newcommand{\calH}{\mathcal{H}}
\newcommand{\calB}{\mathcal{B}}
\newcommand{\calM}{\mathcal{M}}
\newcommand{\abs}[1]{\lvert#1\rvert}
\newcommand{\norm}[1]{\lVert#1\rVert}

\usepackage{float}

\author{
{\sf David Gamarnik\thanks{Operations research Center and Sloan School of Management, MIT. Email: gamarnik@mit.edu} 
and Aukosh Jagannath\thanks{Department of Statistics and Actuarial Sciences and Department of Applied Mathematics, 
University of Waterloo. Email:a.jagannath@uwaterloo.ca}}
}

\begin{document}

\title{The Overlap Gap Property and Approximate Message Passing Algorithms for $p$-spin models}
\date{\today}

\maketitle

\begin{abstract}
We consider the algorithmic problem of finding a near ground state (near optimal solution) of a $p$-spin model. We show that for a class of algorithms
broadly defined as Approximate Message Passing (AMP), the presence of the Overlap Gap Property (OGP), appropriately defined,
is a barrier.
We conjecture that when $p\ge 4$ the model does indeed exhibits OGP (and prove it for the space of binary solutions). 
Assuming the validity of this conjecture, as an implication, the AMP fails to find near ground states in these models, per our result. We extend our result
to the problem of finding pure states by means of Thouless, Anderson and Palmer (TAP) based iterations, which is yet another example
of AMP type algorithms. We show that such iterations fail to find pure states
approximately, subject to the conjecture that the space of pure states exhibits the OGP, appropriately stated, when $p\ge 4$.
\end{abstract}


\section{Introduction}
Given an $N$-tensor $A=(A_{i_1,\ldots,i_p}, 1\le i_1,\ldots,i_p\le N)\in (\R^N)^{\otimes p}$ 
of order $p$ and an $N$-vector $u\in\R^N$, define the usual inner tensor product by
\begin{align}\label{eq:A(u)}
A(u)\triangleq \sum_{1\le i_1,\ldots,i_p\le N}A_{i_1,\ldots,i_p}u_{i_1}\cdots u_{i_p}.
\end{align}
Consider the associated normalized variational problem over the binary cube $\calB_N\triangleq \{-1,1\}^N$:
\begin{align}\label{eq:GroundState}
\eta_N\triangleq {1\over N}\min_{\sigma\in \calB_N}A(\sigma).
\end{align}
The case when $A$ consists of i.i.d. zero mean Gaussian random entries with variance $1/N^{p-1}$ that is $\mathcal{N}\left(0,{1\over N^{p-1}}\right)$
corresponds to the problem of finding a ground state of a $p$-spin model with Gaussian couplings and the (unique) vector $u^*$ achieving the minimization
value is called the ground state~\cite{panchenko2013sherrington}. The choice of variance $1/N^{p-1}$ and the normalization $1/N$ is dictated by the associated
Gibbs distribution defined by assigning probability weight proportional to $\exp\left(-\beta A(\sigma)\right)$ to each $\sigma\in \calB_N$ for some fixed inverse
temperature parameter $\beta\in \R_+$. In this case the partition function 
\begin{align*}
Z\triangleq \sum_{\sigma\in \calB_N}\exp(-\beta A(\sigma))
\end{align*}
is well approximated by $N\eta_N$ as $\beta$ increases and $\eta_N$ is known to converge to a strictly negative limiting 
value $\eta^*<0$ with high probability (w.h.p.)
as $N\to\infty$.
For us though the details of the choice of scaling are immaterial and the  variational
problem above is equivalent to the case when $A$ consists of i.i.d. standard normal entries and the normalization $1/N$ is skipped. Another standard
assumption in the literature is to assume a symmetry of $A$, for example assuming that entries are fully determined by i.i.d. entries
corresponding to $i_1\le \cdots \le i_p$. This difference is again immaterial. Indeed, consider the tensor $\bar A$ defined by
\begin{align*}
\bar A_{i_1,\ldots,i_p}={1\over p!}A^\pi,
\end{align*}
where $A^\pi$ is defined by
\begin{align*}
A^\pi_{i_1,\ldots,i_p}=A_{\pi(i_1,\ldots,i_p)}, \qquad 1\le i_1,\ldots,i_p\le N,
\end{align*}
for any permutation $\pi$  of  $1,\ldots,p$. Note that $\bar A$ is symmetric and satisfies $\bar A(u)=A_\pi(u)$ for every $\pi$.

In the present paper we focus on the algorithmic question of solving the minimization problem (\ref{eq:GroundState}) approximately and
efficiently (in polynomial time). That is, the question is one of existence of a polynomial time algorithm which for every $\epsilon>0$ 
produces a sequence of solution $\sigma_N\in \calB_N$ satisfying 
\begin{align*}
A(\sigma_N)/N\le (1-\epsilon)\eta^*,
\end{align*}
as $N\to\infty$, ideally w.h.p. as $N\to\infty$. This problem was successfully solved recently by Montanari~\cite{montanari2018optimization} 
in the case of the Sherrington-Kirkpatrick
model, which is the special case corresponding to $p=2$. The  result though assumes the validity of 
a (widely believed) conjecture that the overlap distribution function is strictly increasing.  
In particular, it  assumes the absence of an interval $[\nu_1,\nu_2]$ inside the support of the overlap distribution with zero mass, namely that
the overlap distribution does not exhibit the Overlap Gap Property (OGP). 
The algorithm is based on a variant of Approximate Message Passing (AMP)
type algorithms, which in the context of spin glasses is well motivated by the so-called Thouless, Anderson and Palmer (TAP) equation describing
the magnetization of spins in spin glass models. AMP, as a class of algorithms was also found to be one of the most effective classes of algorithm in many models of 
signal processing~\cite{kabashima2003cdma,donoho2009message,bolthausen2014iterative,bayati2011dynamics,javanmard2013state,
bayati2015universality,berthier2017state}, specifically models involving a ''planted signal'' (which does not apply to our $p$-spin model).
The algorithmic result of~\cite{montanari2018optimization} in its order was inspired by a similar result by Subag~\cite{subag2018following} 
regarding the problem of finding a near 
ground states in a spherical``mixed $p$-spin''  model. Here one considers a linear combination of objectives of  the form \eqref{eq:A(u)} as one varies $p$,
with the coefficients being fixed, and optimizing over the unit sphere $\{u:\|u\|_2 \leq 1\}$ instead of $\calB_N$.
Here a polynomial time constrictions of near optimal solutions is provided, under the assumption that the model does not exhibit OGP
(see part (2) of Proposition~1 in~\cite{subag2018following}). For the case of spherical models,
the necessary and sufficient conditions for the OGP are known~\cite{ChenSen15,JagTobLT16,TalSphPF06}. Both $p$-spin and spherical $p$-spin models
are related to the Random Energy Model (REM) considered from the algorithmic perspective by Addario-Berry  and Maillard~\cite{addario2018algorithmic}, 
where in contrast to~\cite{montanari2018optimization} and \cite{subag2018following} algorithmic hardness is established away from the ground state value.
One should note, however that REM is an oracle based optimization problem and thus does not fit classical input size based algorithmic complexity 
questions arising in the context of $p$-spin and spherical $p$-spin models.

At the same time it is known that the OGP does take place in $p$-spin models when $p\ge 4$, as was established in~\cite[Theorem 3]{chen2019suboptimality}.
In particular, it was shown that for every even $p\ge 4$, there exists $\mu>0, 0<\nu_1<\nu_2<1$ such that w.h.p.
for every pair of  solutions $\sigma_1,\sigma_2$ satisfying $A(\sigma_j)/N\le \eta^*+\mu$ for $j=1,2$, the associated normalized overlap, defined simply
as the normalized absolute value of the 
inner product $(1/N)|\langle \sigma_1,\sigma_2\rangle|$ is either at most  $\nu_1$ or at least some  $\nu_2$.
 This naturally raises the question as to whether the OGP creates a barrier to the success of AMP when $p\ge 4$. 
The main result of our paper, Theorem~\ref{theorem:Main-result}, 
is to establish precisely this fact under the assumption that a certain relaxed version of the OGP takes place when $p\ge 4$. 
The relaxed version concerns the optimization problem $\min A(u)$ when $u$ is relaxed to be in Hilbert cube $u\in [-1,1]^N\triangleq \calH_N$, and otherwise
is defined in the same way  as for $\calB_N$. This relaxed version would be a rather straightforward implication of the OGP for binary solutions
if one could show that every nearly optimal solution in $\calH_N$ is nearly binary. Unfortunately, even this fact is not known, and we leave it as an interesting,
though, as we believe, an approachable open problem. As a consequence of our main result, we show that extension of the AMP result 
of~\cite{montanari2018optimization} to the case $p\ge 4$ is not possible. As another implication, we show that a natural iterative scheme of computing the fixed
point of the TAP equations fails as well in the case $p\ge 4$. We note, that this iterative scheme is known to succeed in the high temperature regime
due to the result of~\cite{bolthausen2014iterative}. Another important class of algorithms ruled out by our negative result is gradient descent type 
algorithms. Since the gradient of $A(u)$ is a linear combination of the vectors of the form $A(\cdot,u)$ defined below in (\ref{eq:A(u)marginal}), 
then a discrete implementation of the gradient descent algorithm in the form $u^t=u^{t-1}+\eta_{t-1}\nabla A(u^{t-1})$ for some step choices $\eta_t$ 
is also a special case of our class of AMP algorithms.

One challenge in establishing our result formally is the formalization
of the class of AMP algorithms to begin with. 
Unfortunately, there is no one formal definition for it, but rather there is a vaguely proposed scheme for a class of iterations inspired by the Belief
Propagation type algorithms. The iterations take  the form $u^{t+1}=F_t(G_t(u^t),G_{t-1}(u^{t-1}),\ldots,G_0(u^0)), t\ge 0$ 
and are performed for some constant number of rounds 
$t=0,1,\ldots,T$,  where $F^t$ is an in general  $t$-dependent 
function involving vector $A(\cdot, u)\in \R^N$ defined by 
\begin{align}\label{eq:A(u)marginal}
A(\cdot, u)\triangleq \left(\sum_{i_2,\ldots,i_p}A_{i,i_2,\ldots,i_p}u_{i_2}\cdots u_{i_p}, 1\le i\le N\right),
\end{align}
for any $u\in\R^N$, 
as well other non-linear operators $G_t:\R^N\to\R^N$ defined typically through some kind of univariate or $t$ variant 
non-linear functions $g_t:\R^t\to \R$ applied coordinate-wise in some way. We note that (\ref{eq:A(u)marginal}) is simply a matrix vector product $Au$ when $p=2$.

Thus, as a first step, we introduce a precise class of such iterative  
algorithms (functions) $F^t$, and  associate it with a precise set of assumptions. We show separately that the algorithm
of~\cite{montanari2018optimization} is  a special case.  We assume that the results $U^t, 0\le t\le T$ of each
iterations are truncated so that the resulting vector always belongs to $\|\cdot\|_\infty$ bounded region of the form 
$[-M,M]^N$ for some constant $M$.
The rational for the truncation is as follows. In the implementation of the AMP the iterations $F^t, 0\le t\le T$
produce a real valued vector $U^T\in \R^N$ which is then projected to a vector in $\calB_N$ in a way discussed below. 
The idea here is  that $U^T$ is a vector which is ''close enough'' to
some vector $\sigma\in \calB_N$ which is a near-ground state. In particular, the typical entries of $U^t$ are ''not too far'' from interval $[-1,1]$, and in particular
are bounded by $M$. We restrict every vector $U^t$ to be in $[-M,M]^N$ for technical convenience.
The rounding scheme $[-M,M]^N\to \calB_N$ assumed to be adopted by our class of AMP 
is similar the one that was used in~\cite{montanari2018optimization}:  first $U^T$ is projected to vector $V\in \calH_N$ 
via a natural truncation $x\to \min(\max(x,-1),1)$, and then some rounding
scheme $\Pi:\calH_N\to \calB_N$ is adopted by the algorithm designer, which is guaranteed asymptotically to never lower the quality of the solution, 
that is, it guarantees that $A(\Pi(V))/N\le A(V)/N+o(1)$.  Our main result, stated more precisely, says that
for \emph{any} AMP algorithm thus defined, the  vector $V$ is w.h.p. sub-optimal, namely  $A(V)/N$ exceeds 
$\eta^*$ by some fixed constant $\mu>0$, w.h.p. as $N\to\infty$. Thus we establish that the vector $V$ obtained in the pen-ultimate (before $\Pi$)
step of the AMP is sub-optimal. We note that it is precisely this vector which is shown to be nearly optimal in the case $p=2$ in the argument
of~\cite{montanari2018optimization}. The last step of converting a real vector $V\in \calH_N$ to $\Pi(V)\in \calB_N$ is just used there in order 
to obtain a genuinely binary vector.
We do not establish that the ultimate vector $\Pi(V)\in \calB_N$ is sub-optimal, and this is a limitation of our technique.
We note however that showing near optimality of $\pi(V)$ without showing near optimality of $V$ 
 would amount to believing that the rounding
$\Pi$ is somehow mysteriously capable of producing near optimal binary solution $\Pi(V)$ from a presumably
far from optimal fractional solution $V\in \calH_N$, which is something which does not 
seem to be plausible, and something which not established in~\cite{montanari2018optimization}. 
Nevertheless, it would be admittedly a more complete result to show directly that $\Pi(V)$ is far from optimal, without assuming the 
same for $V$, but we are  currently unable to make this argument
rigorous, and leave it for further investigation.

\subsubsection*{Proof of the main result. Outline} We now describe the main ingredients of our proof. 
First, as a consequence of a result established in~\cite{chen2019suboptimality}, we show that the OGP holds w.h.p. not just for one instance of $A$,
but for a continuous family of sets of the form $\mathcal{A}=\left(\sqrt{1-\tau}A+\sqrt{\tau}\hat A\right), \tau\in [0,1]$ where $\hat A$
is an independent instance of $A$. Note that for each fixed $\tau$, the corresponding tensor has the same distribution as $A$.
An easy consequence of the OGP result in ~\cite{chen2019suboptimality} and the chaos result of~\cite{chen2018energy}, is the fact, which 
we prove in this paper (Theorem~\ref{theorem:OGP-for-tensors}), 
that the OGP holds for $\mathcal{A}$ as well in the sense that for any two
$A_{\tau_1},A_{\tau_2}\in\mathcal{A}$ and any $\sigma_1,\sigma_2\in\calB_N$ satisfying 
$A_{\tau_j}(\sigma_j)/N\le \eta^*+\mu$, it 
is again the case that $N^{-1}|\langle \sigma_1,\sigma_2\rangle|\in [0,\nu_1]\cup [\nu_2,1]$, for the same values $\nu_1,\nu_2,\mu$. 
The chaos property roughly speaking  says for any fixed $\tau_1\ne\tau_2$ near optimal solutions of $A_{\tau_1}$ and $A_{\tau_2}$
are nearly orthogonal, see Theorem~\ref{theorem:chaos} below. Our main conjecture regarding OGP 
(Conjecture~\ref{conjecture:OGP-for-tensors-Hilbert-cube}), which we use as an 
assumption of the main result, is the conjecture that OGP
holds in fact for near optimal solutions in $\calH_N$ as opposed to those in $\calB_N$, for the same family of instances. Establishing
this conjecture is an interesting open question.

Our main ingredient of the proof is then to show that the iterations $U^T=U^T(A)$ as function of $A$ are sufficiently ''continuous'' to perturbation of the entries
of $A$. Specifically, we obtain an upper bound on $N^{-1}\|U^T(A_\tau)-U^T(A_0)\|_2$ for the interpolation scheme $\mathcal{A}$ which 
is sufficiently continuous in $\tau$. This result is the subject of Theorem~\ref{theorem:iteration-bounds}.
A straightforward implication is that the same bound applies to $N^{-1}\|V(A_\tau)-V(A_0)\|_2$, where, as we recall $V(A)$ is projection 
of $U^T(A)$ through the truncation $x\to\min(\max(x,-1),1)$. 
Separately, we use the independence of $A$ and $\hat A$ to argue the near orthogonality of $V(A)$ and $V(\hat A)$. The continuity result above then is used
to show  that for an appropriate choice of  $\tau$, it will hold  that $N^{-1}\langle V(A_\tau),V(A)\rangle \in (\nu_1,\nu_2)$. 
The (conjectured) OGP
property implies that this choice of $\tau$ corresponds to a sufficiently sub-optimal solution $V(A_\tau)$, which contradicts concentration property
of $A(V(A_\tau))$, which we establish separately using standard techniques, including Gaussian concentration of measure and Kirszbraun's Theorem.

\subsubsection*{Prior results on OGP and algorithmic implications} The concept of OGP originates in the study of spin glass models, specifically
the study of overlap distribution of replicas generated according to some associated Gibbs distribution, such as the one described above. Understanding
the limiting distribution of overlaps is of an utmost importance to spin glass theory and has recieved significant attention \cite{AuffChen13,auffinger2017sk,JagTobPD15,JagTob16boundingRSB}. 
The first connections between to study of the overlaps
and algorithms were made in the context of random constraint satisfaction problems, such as random K-SAT problem and many other similar problems.
These problems exhibit an ''infamous'' gap between the range of parameters for which a satisfying assignment exists vs those for which solutions can be found
in polynomial time. The apparent hardness was linked conjecturally to the clustering (shattering) property of these models which were discovered to
appear roughly in the regime
where known polynomial time algorithms 
fail~\cite{achlioptas2008algorithmic,AchlioptasCojaOghlanRicciTersenghi,mezard2005clustering,
mezard2002analytic,coja2011independent}. The clustering
property says, roughly speaking, that a large  part of the  set of satisfying assignments can be partitioned into clusters separated by 
Hamming distance, which is of the order of the size of the model itself. It is notable that the proof technique used to establish such a clustering property
actually shows something more: the overlaps between pairs of typical (random in some appropriate sense) satisfying assignments lie in a disconnected
union of intervals  $[0,\nu_1]\cup [\nu_2,1]$. Thus the set of solutions is disconnected not only with respect to its ambient metric space, 
but also with respect to its one-dimensional projection onto  the set of possible overlap values. 
The proof technique relies on fairly standard application of the moment
method.  The disconnectivity of overlaps  (that is  the presence of the overlap gaps of the form $(\nu_1,\nu_2)$) 
was later used  as  an obstruction to a class of local
algorithms defined as so-called Factors of IID in~\cite{gamarnik2014limits,rahman2017local,gamarnik2017performance,chen2019suboptimality},
and for random walk
type algorithms in~\cite{coja2017walksat}. 
It is this line of work which the closest in spirit to the present one, as one can think of AMP as a natural definition
for ''local'' algorithms defined on dense instances -- instances not defined on sparse graphs and hypergraphs.

It is important to note that  while OGP implies the clustering property, the converse in general is not true. 
Indeed if the OGP takes place then  one can partition the set of all solutions of interest into those which have
overlap at least $\nu_2$ with some arbitrarily marked solution $\sigma^o$ (thus marked ''Cluster 1''), 
vs solutions with overlap at most $\nu_1$ with $\sigma^o$ (thus marked ''other clusters''), leading to a set of at least two clusters separated by 
a significant distance. On the other hand, one can easily create 
a subset of $\calB_N$ for which the set of all overlaps  spans the entire interval $[0,1]$, though at the same time admits clustering partition.

The OGP was further established for some other models, some involving planted 
signals~\cite{gamarnik2018finding,david2017high,gamarnik2019overlap}. It was shown in~\cite{gamarnik2019overlap} to be an obstruction
to Glauber Dynamics type algorithms by showing that OGP implies the existence of a free energy well, a property which was shown to be a barrier
for Markov chain type algorithms in problems involving planted signals~\cite{arous2018algorithmic}. 
A related notion of free energy barriers associated with these gaps were also shown to be obstructions for local Markov chain type algorithms for
problems of the class considered here in \cite{BAJag17}, where it was also shown
that these free energy barriers occur in a broad class of models including both the $p$-spin and spherical $p$-spin models. 
It can be shown that OGP implies  the existence of a free energy barrier at sufficiently low temperatures.
It is of interest to establish the broadest
class of algorithms for which OGP is a provable barrier.

The remainder of the paper is structured as follows. In the next section we introduce the formalism of the AMP algorithms. 
In Section~\ref{section:OGP} we give the definition of the OGP, state the corresponding conjecture and state our main result.
The validity of OGP for binary solutions is proven in the same section. Some preliminary technical results are established
in Section~\ref{section:preliminary}. Our main technical result is Theorem~\ref{theorem:iteration-bounds} which is stated and proven 
in Section~\ref{section:Iterations-and-bounds}. We note that it is a purely deterministic result showing that the output of the AMP
depends on the values of the tensor $A$ sufficiently continuously. In Section~\ref{eq:concentration-of-V} we establish the concentration
property of the solution $V$ produced by the AMP around its expectation. Our main theorem is proven in Section~\ref{section:OGP-obstructs-AMP}.
In Section~\ref{section:TAP-iteration} we consider TAP solutions and show that a natural class of iterations suggested by TAP fails to find the fixed point of TAP,
modulo the same Conjecture~\ref{conjecture:OGP-for-tensors-Hilbert-cube}, since the iterations are a special case of the class of AMP algorithms
we define. This result is a direct implication of our main result, Theorem~\ref{theorem:Main-result}. It contrasts with the positive result
of Bolthausen~\cite{bolthausen2014iterative}, which establishes that these iterations do converge to the solution of TAP equations in the high-temperature
setting. In Section~\ref{section:AMP-p=2} we verify that the AMP algorithm constructed in~\cite{montanari2018optimization}
also fit the general definition of AMP introduced in this paper. Finally, we conclude in Section~\ref{section:Conclusions} where we state some open questions.

\section{Approximate Message Passing iterations formalism}\label{section:AMP-formalist}
$\langle x,y\rangle=\sum_{1\le i\le N}x_iy_i$ denotes inner product of vectors $x,y\in R^N$. For any  tensor $B\in (\R^N)^{\otimes p}$
$\|B\|_2$ denotes the Frobenius norm $\sqrt{\sum_{1\le i_1,\ldots,i_p\le N}B_{i_1,\ldots,i_p}^2}$, and $\|B\|_{\rm op}$ denotes
the operator norm
\begin{align*}
\max_{u_1,\ldots,u_p}B(u_1,\ldots,u_p)
\end{align*}
where the maximum is over all $u_1,\ldots,u_p\in \R^N, \|u_j\|_2=1, 1\le j\le p$. By Cauchy-Schwartz inequality
$\|B\|_{\rm op}\le \|B\|_2$.

Throughout the paper
$A\in (\R^{N})^{\otimes p}$ denotes $N$-size order $p$ tensor consisting of $\mathcal{N}(0,N^{-(p-1)})$ i.i.d. entries. For any $u_1,\ldots,u_{p-1}\in\R^N$ 
let
\begin{align*}
A(u_1,\ldots,u_p)=\sum_{1\le i_1,\ldots,i_p\le N}A_{i_1,i_2,\ldots,i_{p-1}}u^1_{i_1}\cdots u^{p}_{i_{p}},
\end{align*}
so that for any $u\in \R^N$, $A(u)=A(u,\cdots,u)$ as in (\ref{eq:A(u)}). Here $u_r=(u^r_1,\ldots,u^r_N)$.
For any $u_1,\ldots,u_{p-1}\in\R^N$ we also introduce 
\begin{align}\label{eq:A-dot-u}
y=A(\cdot,u_1,\ldots,u_{p-1})\in \R^N
\end{align} 
defined by
\begin{align*}
y_i=\sum_{1\le i_1,\ldots,i_{p-1}\le N}A_{i,i_1,\ldots,i_{p-1}}u^1_{i_1}\cdots u^{p-1}_{i_{p-1}} \qquad 1\le i\le N.
\end{align*}
Similarly, for any $u\in\R^N$ we write $A(\cdot,u)$ instead of $A(\cdot,u,u,\ldots,u)$ for short. 
We recall the definition of $\eta_N$ from (\ref{eq:GroundState}).
Observe that we may view $A(u)$ as a centered Gaussian process indexed by $\calH_N$, which
has covariance 
$$
\E{A(u)A(v)}= N\Big(\frac{\langle u, v\rangle}{N}\Big)^p.
$$
In particular, $|\E{A(u)A(v)}|\le N$ for any $u,v\in\R^N$ with $\|u\|_2,\|v\|_2\le 1$. The following concentration result is then an immediate consequence of the Borell-TIS inequality, Theorem 2.1.11 of~\cite{adler2009random}.

\begin{theorem}\label{theorem:existence-plus-concentration}
For every $\delta>0$ 
\begin{align*}
\pr\left(|\eta_N-\E{\eta_N}|\ge \delta\right)\le \exp(-(1/4)\delta^2 N),
\end{align*}
for all sufficiently large $N$.
\end{theorem}
A major consequence of the development in spin glass theory is the existence of the limit
\begin{align}\label{eq:gamma-limit}
\lim_{N\to\infty}\E{\eta_N}=\eta^*<0,
\end{align}
which by Theorem~\ref{theorem:existence-plus-concentration} also implies that the limit $\eta_N\to\eta^*$ holds w.h.p. as $N\to\infty$.

We now introduce a set of assumptions which are used to define a class of AMP algorithms. Fix a positive integer $T$ and 
an $M>0$.
Consider two sequences of functions $f_t:[-M,M]^t\to\R$ and 
$F_t:\R\times[-M,M]^{t} \to\R, 1\le t\le T$.  

\begin{Assumption}\label{assumptions:F}
$f_t(0)=0$. Furthermore,
functions $f_t,F_t$ are Lipschitz continuous on their respective domains. 
More precisely, there exists $\zeta\in\R_+$ such that for all $1\le t\le T$,
\begin{align}\label{eq:f-lipschitz}
\sup_{u,v\in[-M,M]^t}|f_t(u)-f_t(v)|&\le \zeta \|u-v\|_2,\\
\sup_{u,v\in\R\times[-M,M]^t} |F_t(u)-F_t(v)|&\le \zeta \|u-v\|_2.\label{eq:F-lipschitz}
\end{align}
\end{Assumption}
The assumption (\ref{eq:F-lipschitz}) says that the function $F_t$ is Lipschitz on an infinite rectangle $\R\times [-M,M]^t$
This will be required due to the special role played by the first variable  of $F_t$.

Fix a positive constant $M>1$. Let $x_M=\max(-M,\min(x,M))$ denote an $M$-truncation for any $x\in \R$. When $x$ is a vector,
$x_M$ is assumed to be applied coordinate-wise. We now define the iterations forming the basis of AMP. 
Fix $U^0\in [-M,M]^N$ and define the sequence $U^t\in\R^N, 1\le t\le T$ as follows
\begin{align}\label{eq:Iteration-U^t}
U^t=\left[F_t(A(\cdot,f_t(U^0,\ldots,U^{t-1})),U^0,\ldots,U^{t-1})\right]_M \in [-M,M]^N,
\end{align}
where $F_t,f_t$ and $M$ are applied component-wise. In other words, in step $t$, first a vector $f_t(U^0,\ldots,U^{t-1})\in\R^N$
is formed by applying $f_t$ coordinate-wise (recall that the domain of $f_t$ is $\R^t$). Then this vector is used to define
vector $A(\cdot,f_t(U^0,\ldots,U^{t-1}))$ via (\ref{eq:A(u)marginal}). This vector is concatinated with prior vectors $U^0,\ldots,U^{t-1}$
to form an $N\times (t+1)$ matrix $\left(A(\cdot,f_t(U^0,\ldots,U^{t-1})),U^0,\ldots,U^{t-1}\right)\in\R^{N\times (t+1)}$. Then
function $F_t$ is applied coordinate-wise. Finally the $M$-truncation
is applied to each of the $N$ coordinates of the  vector thus obtained, resulting in $U^t$.

We now describe an algorithm which uses AMP to generate a solution in $\calB_N$. For this purposes we assume that the algorithm designer
has access to some (computable) projection function $\Pi_N:\calH_N\to \calB_N$.  We discuss this further below.

\begin{algorithm}[AMP Algorithm]\label{alg:AMP}

The  algorithm is parametrized by $U^0,M,T,(f_t, 1\le t\le T),(F_t, 1\le t\le T),\Pi_N$.

\begin{enumerate}
\item[]
Input  $A\in (\R^N)^{\otimes p}$.

\item[{\rm \textbf{Step 1}}]
Compute $U^T$ using (\ref{eq:Iteration-U^t}). 

\item[{\rm \textbf{Step 2}}]
Project $U^T$ to $\calH_N$ by applying
transformation $x\to [x]_1=\max(\min(x,1),-1)$, coordinate-wise. Denote the resulting vector by $V\in \calH_N$.

\item[{\rm \textbf{Step 3}}]
Output $\sigma=\Pi(V)\in \calB_N$.

\end{enumerate}

\end{algorithm}

In some sense the details of the projection $\Pi_N$ are immaterial to us since our negative result will be concerned with the quality 
of the solution $V$ itself and not its projection. Nevertheless, for completeness we describe now the projection used in (\cite{montanari2018optimization}),
which we denote by $\Pi_N^{\sign}$. 
The projection was defined only for $p=2$ which was the case of interest. But it is straightforward to extend the idea to the case of general $p$.
Set $z^{(0)}=V$.
For $j=1,\ldots,N$, construct $z^{(j)}$  by making
all coordinates $\ell\ne j$ of $z^{(j)}$ to be the same as of $z^{(j-1)}$, and setting the $j$-th coordinate of $z^{(j)}$ to 
be the sign opposite of 
\begin{align}\label{eq:linear-multiplier-V_j}
\sum_{j\ne i_1\ne i_2\ne\cdots\ne i_{p-1}}A_{j,i_1,i_2,\ldots,i_{p-1}}z^{(j-1)}_{i_1,\ldots,i_{p-1}}.
\end{align}
In particular, the first $j$ coordinates of $z^{(j)}$ are $\pm 1$, but the remaining coordinates are real valued in general.
Set $\Pi_N^{\sign}(V)=z^{(N)}$.

Let us comment on the meaning and motivation behind the steps of the AMP algorithm above and also the motivation behind the 
projection $\Pi_N$ described above and used in (\cite{montanari2018optimization}). The idea is that when the AMP algorithm succeeds,
the vector $V$, while not being an element of the binary cube $\calB_N$, should be nearly optimal in the sense that
\begin{align*}
A(V)\approx \inf_{w\in \calB_N}A(w),
\end{align*}
and should not be too far from $\calH_N$, so that the projecting $U^T$ to $V\in \calH_N$ does not change the objective value significantly.
That is $A(V)\approx A(U^T)$. Next one observes that $\Pi_N^{\sign}$ effectively rounds $V$ to a vector $z^{(N)}$ in $\calB_N$ 
in such a way that the objective value only decreases
asymptotically. This is verified by observing that for each coordinate $j$, the dependence of $A(V)$ on variable $V_j$ is linear in $V_j$, except for
terms $A_{i_1,\ldots,i_p}$ with repeating coordinates (i.e. such that $i_\ell=i_r$ for some $1\le \ell\ne r\le p$). Since $V\in \calH_N$ and
thus $|V_j|\le 1$,  the linearity allows to round $V_j$ to $-1$ or $1$ while only decreasing the objective value. This is done trivially
by setting $V_j$ to be the sign opposite of the one of the multiplier of $V_j$, which is (\ref{eq:linear-multiplier-V_j}). This is done
iteratively over all $N$ coordinates. The terms corresponding to repeating coordinates are easily shown to be of lower order of magnitude
than the objective value. As a result
one obtains a vector $z=z^{(N)}\in \calB_N$ satisfying
\begin{align*}
A(z)\lesssim A(V)\approx\inf_{\sigma\in \calB_N}A(\sigma).
\end{align*}
But since $z$ belongs to the solution space itself (the binary cube $\calB_N$), it must be the case that in fact
\begin{align*}
A(z)\approx A(V)\approx \inf_{\sigma\in \calB_N}A(\sigma),
\end{align*}
and thus the success of AMP is validated.
Importantly, the near optimality of $z$ is argued from the near optimality of $V$ itself.
This discussion is of key essence to the main result of our work, which is stated in the next section.

\section{The OGP conjecture and the main result}\label{section:OGP}
Consider an arbitrary set $\mathcal{A}$ of tensors $A\in (\R^N)^{\otimes p}$.

\begin{Defi}
The set $\mathcal{A}$ satisfies the Overlap Gap Property (OGP) with domain $\mathcal{S}_N\subset \R^N$, 
and parameters $\mu>0, 0<\nu_1<\nu_2<1$ if for every pair
$A_j\in \mathcal{A}, j=1,2$ and every $u_j, j=1,2$ satisfying 
\begin{align*}
{1\over N}A_j(u_j)\le {1\over N}\inf_{w\in \mathcal{S}_N}A_j(w)+\mu, \qquad j=1,2,
\end{align*}
it holds
\begin{align}\label{eq:OGP}
{|\langle u_1,u_2\rangle|\over \|u_1\|_2\|u_2\|_2} \in [0,\nu_1]\cup [\nu_2,1].
\end{align}
\end{Defi}
Namely, every pair of nearly ($\mu$-close) optimal solutions with respect to any two members of $\mathcal{A}$ cannot have normalized inner product
in the interval $(\nu_1,\nu_2)$.

Consider two independent random tensors $A$ and $\hat A$ in $(\R^N)^{\otimes p}$ both with i.i.d. $\mathcal{N}\left(0,1/N^{p-1}\right)$ entries.
Introduce the interpolated set of tensors $A_\tau\triangleq\sqrt{1-\tau}A+\sqrt{\tau}\hat A$ with $\tau$ varying in $[0,1]$. 
Note that for each fixed $\tau$, $A_\tau$ is distributed as $A$. Our main conjecture regarding the OGP concerns the set
$\mathcal{A}\triangleq (A_\tau, 0\le \tau\le 1)$.

\begin{conjecture}\label{conjecture:OGP-for-tensors-Hilbert-cube}
For every even $p\ge 4$ here exists $\mu>0, 0<\nu_1<\nu_2<1$ such that $\mathcal{A}$ 
described above satisfies the OGP with domain $\mathcal{S}_N=\calH_N$ and  parameters $\mu,\nu_1,\nu_2$, with probability at least
$1-\exp(-cN)$, for some $c>0$ for all sufficiently large $N$. Furthermore, for every $\delta>0$ and every $v_1,v_2\in \calH_N$
satisfying $A(v_1)/N\le (1-\delta)\E{\eta_N}, \hat A(v_2)/N\le (1-\delta)\E{\eta_N}$, it holds $|\langle v_1,v_2\rangle|\le \delta N$
with probability at least $1-\exp(-cN)$ for some $c>0$ and all large $N$.
\end{conjecture}

Our  main result stated below assumes the validity of this conjecture. To state this result, let us introduce 
the following. Let $\calM_1([-M,M]^N)$ denote the space
of probability measures on $[-M,M]^N$. Let $V(A,T,U^0)$ denote the output of the first two steps
of Algorithm \ref{alg:AMP} after $T$ steps with coefficient matrix $A$ and initial data $U^0$,
where the entires of $A\in (\R^N)^{\otimes p}$ are i.i.d. $\mathcal{N}(0,N^{-(p-1)})$. Then the following holds.

\begin{theorem}\label{theorem:Main-result}
Let $p\geq 4$ be even. Let $M\geq 1$ and $\zeta>0$. Assume that $(f_t),(F_t)$ satisfy Assumption~\ref{assumptions:F} with Lipschitz constant $\zeta$
and that Conjecture~\ref{conjecture:OGP-for-tensors-Hilbert-cube} holds. Then there exists $\bar\mu>0$
and $c>0$, such that for $N$ sufficiently large and any $\nu \in \calM_1([-M,M]^N)$, if $U^0\sim \nu$,
then $V=V(A,T,U^0)$ satisfies
\[
\pr\left({A(V)\over N}\ge {\min_{\sigma\in \calB_N} A(\sigma) \over N}+\bar\mu\right) \geq 1-\exp(-cN).
\]
\end{theorem}

Thus we argue the failure of the AMP to find a vector $V\in \calH_N$ which is a near optimizer of $A$. As discussed  earlier, 
this is a negative result regarding the performance of AMP, since 
finding such near optimal $V$ is a key step towards finding a near optimal member $z$ of the binary cube $\calB_N$. Ideally, one would establish
that the vector $z$ obtained from $V$ via any projection scheme, such as the one  
described above is also $\mu$-away from optimality. Unfortunately, our proof technique stops short
of that due to the potential sensitivity of the sign function used on obtaining $z$ to perturbation of $A$, thus 
potentially violating stability used crucially in the proof of our main result. We leave this as an interesting open question.

A partial support to the validity of Conjecture~\ref{conjecture:OGP-for-tensors-Hilbert-cube} 
above is its validity for the domain $\mathcal{S}_N=\calB_N$ as we now establish.

\begin{theorem}\label{theorem:OGP-for-tensors}
For every even $p\ge 4$ here exists $\mu>0, 0<\nu_1<\nu_2<1$ such that $\mathcal{A}$ 
described above satisfies the OGP with domain $\mathcal{S}=\calB_N$ and  parameters $\mu,\nu_1,\nu_2$, with probability at least
$1-\exp(-cN)$, for some $c>0$ for all sufficiently large $N$. Furthermore, for every $\delta>0$ and every $\sigma_1,\sigma_2\in \calB_N$
satisfying $A(\sigma_1)/N\le (1-\delta)\E{\eta_N}, \hat A(\sigma_2)/N\le (1-\delta)\E{\eta_N}$, it holds $|\langle \sigma_1,\sigma_2\rangle|\le\delta N$
with probability at least $1-\exp(-cN)$ for some $c>0$ and all large $N$.
\end{theorem}

\begin{proof}
We note that in the case $\mathcal{S}_N=\calB_N$, since $\|\sigma\|_2=N$ for each $\sigma\in \calB_N$, 
the requirement (\ref{eq:OGP}) in definition of OGP simplifies to 
\begin{align*}
{|\langle \sigma_1,\sigma_2\rangle|\over N} \in [0,\nu_1]\cup [\nu_2,1].
\end{align*}
It was established in~\cite{chen2019suboptimality}, Theorem~3 that the OGP holds for a single instance of a tensor $A$, i.e. $\mathcal{A}=\{A\}$,
with probability at least $1-\exp(-cN)$ for some $c>0$ and all large $N$.
At the same time the following chaos property was established in~\cite{chen2018energy}:

\begin{theorem}[\cite{chen2018energy}, Theorem~2]\label{theorem:chaos}
For every $\epsilon>0$ and $\tau\in (0,1)$ there exists ${C,}\tilde\mu>0$ such that with probability $1-\exp(-C N)/C$,
for every $\sigma_1,\sigma_2\in \calB_N$ satisfying 
$A(\sigma_1)/N\le \E{\eta_N}+\tilde\mu, A_{\tau}(\sigma_2)/N\le \E{\eta_N}+\tilde\mu$ it holds $|\langle \sigma_1,\sigma_2\rangle|\le \epsilon N$.
\end{theorem}

We now combine these two results. We first claim that it suffices to establish OGP for a discrete finite subsets. Fix $\delta>0$
such that $1/\delta$ is an integer 
and consider $A_\tau$ for $\tau=0,\delta,2\delta,\ldots,\delta(1/\delta)$. We assume OGP holds for this set for some $\mu,\nu_1,\nu_2$ for every
sufficiently small such $\delta$. Now 
for any $\sigma\in \calB_N$
\begin{align*}
A_\tau(\sigma)-A(\sigma)=(\sqrt{1-\tau}-1)A(\sigma)+\sqrt{\tau}A(\sigma).
\end{align*}
In light of concentration bound of Theorem~\ref{theorem:existence-plus-concentration}, for any $\epsilon>0$ we can find small enough $\delta$ so that 
\begin{align*}
\max_{0\le k\le (1/\delta)-1}\sup_{0\le\tau\le\delta}\max_{\sigma\in \calB_N}|A_{k\delta+\tau}(\sigma)-A_{k\delta}(\sigma)|\le \epsilon N
\end{align*}
with probability at least $1-\exp(-cN)$, for some $c>0$
and large $N$. This means that modulo exponentially small probability, every $\sigma$ satisfying $A_{k\delta+\tau}(\sigma)/N\le \E{\eta_N}+\epsilon$
also satisfies $A_{k\delta}(\sigma)/N\le \E{\eta_N}+2\epsilon$. Thus if $\epsilon<\mu-2\epsilon$ then the set
$(A_\tau,\tau\in [0,1])$ satisfies OGP with $\hat\mu=\mu-2\epsilon>0$ and the same $\nu_1,\nu_2$, 
provided that the discrete set $(A_{k\delta}, 0\le k\le 1/\delta)$ satisfies OGP with $\mu,\nu_1,\nu_2$.
Thus we now prove OGP for this discrete set.

Let $\mu,\nu_1,\nu_2$ be OGP parameters for a single instance $A$. By the union bounds over $k=0,1,\ldots,1/\delta$,
the OGP holds for each $A_{k\delta}$ modulo exponentially small probability.
Fix $\delta>0$. 
Applying Theorem~\ref{theorem:chaos}, we find $\tilde\mu$ so that the theorem claim holds for $\epsilon=\nu_1$ and $\tau=\delta$. 
By union bounds this also holds for all pairs $A_{k_1\delta},A_{k_2\delta}, k_1\ne k_2$ modulo exponentially small probability.
Then OGP holds for $\bar\mu\triangleq \min(\tilde\mu,\mu),\nu_1,\nu_2$ by considering separately the cases $k_1=k_2$ and $k_1\ne k_2$, where in the latter
case for every $\sigma_j,j=1,2$ satisfying $A_{k_j\delta}(\sigma_j)/N\le \E{\eta_N}+\bar\mu, j=1,2,$ we simply have 
$|\langle\sigma_1,\sigma_2\rangle|\le \nu_1 N$. 

The second part of the theorem follows immediately from the chaos property of Theorem~\ref{theorem:chaos} in the special case $\tau=1$.
\end{proof}

Conjecture~\ref{conjecture:OGP-for-tensors-Hilbert-cube} would follow from Theorem~\ref{theorem:OGP-for-tensors} if we could establish
that every nearly optimal solution in $\calH_N$ is actually close to a point in $\calB_N$. This is quite plausible as one does not expect
nearly optimal solutions to exist ''deep'' inside the Hilbert cube $\calH_N$.
Unfortunately, we are not able to show this and thus
state it as an interesting open problem.

\begin{conjecture}\label{conjecture:near-binary}
Suppose the entries of $A\in (\R^N)^{\otimes p}$ are generated i.i.d. according to $\mathcal{N}(0,N^{-{(p-1)}})$. For every 
$\epsilon>0$ there exists $\delta>0$ such that with probability at least $1-\exp(-cN)$ for some $c>0$ and large enough $N$,
every $u\in \calH_N$ satisfying $A(u)/N\le (1-\delta)\E{\eta_N}$ also satisfies $\min_{v\in \calB_N}\|u-v\|_2\le \epsilon \sqrt{N}$.
\end{conjecture}

\section{Preliminary technical results}\label{section:preliminary}
In this section we establish several preliminary results.  We begin by recalling the following operator norm bound
\begin{lemma}
\label{lemma:op-norm-bound} There exist constants  $C,c>0$ ,
such that 
\begin{align*}
\pr(\norm{A}_{\rm op} & >CN^{1-p/2})\leq e^{-cN},
\end{align*}
for all sufficiently large $N$.
\end{lemma}

\begin{proof}
The proof of this result is verbatim that from \cite[Lemma 4.7]{arous2018bounding}. We include this for completeness. 

Let $\mathbb{S}^{N}=\{x:\norm{x}_{2}=1\}$ denote the unit $\ell_{2}$-
ball. We may then view $A$ as a centered Gaussian process on $(\mathbb{S}^{N})^{\times p}$,
with covariance
\[
\E{A(x_{1},..,x_{p})A(y_{1},...,y_{p})}=\frac{1}{N^{(p-1)}}\prod_{1\le i\le p} \langle x_{i},y_{i}\rangle.
\]
This process is rotationally invariant.
Fix an $\epsilon>0$, let $\Sigma_{\epsilon}$ denote an $\epsilon-$net
for $\mathbb{S}^{N}$  with respect to $\|\cdot\|_2$ norm,  and let $\Sigma_{\epsilon}^{p}$ denote is $p$-fold
cartesian product. By multilinearity of $A$, 
\[
\norm{A}_{\rm op}\leq\sup_{(x_{1},...,x_{k})\in\Sigma_{\epsilon}^{p}}A(x_{1},...,x_{p})+\epsilon p \norm{A}_{\rm op}.
\]
If we choose $\epsilon$ so that $2p\epsilon\leq1$, we have 
\begin{align*}
\pr\left(\norm{A}_{\rm op}>\lambda\right) & \leq\pr\left(\cup_{\mathbf{x}\in\Sigma_{\epsilon}^{p}}\left\{ A(x_{1},...,x_{p})\geq\lambda/2\right\} \right).
\end{align*}
To bound the right hand side, note that for any fixed $(x_{1},..,x_{p})\in(\mathbb{S}^{N})^{p},$
$A(x_{1},...,x_{p})$ is a centered Gaussian with variance $N^{-p+1},$
so that 
\[
\mathbb{P}(A(x_{1},...,x_{p})\geq\lambda N^{1-p/2})\leq e^{-N\lambda^{2}/2}.
\]
where in the second line, $(x_{1},...,x_{p})$ is any point in $\Sigma^{p}$,
$\abs{\Sigma^{k}}$ denotes its cardinality, and the final inequality
comes from a Gaussian tail bound since $A(x_{1},..,x_{p})$ 
is a centered Gaussian with variance $N^{-p+1}$ . Note furthermore
that we may choose this net so that $\abs{\Sigma_{\epsilon}}\leq(4/\epsilon)^{N},$
\cite[Lemma 5.1]{vershynin2018high}. Thus by rotation invariance and a
union bound, we see that 
\[
\pr(\norm{A}_{\rm op}\geq\lambda N^{1-p/2})\leq\left(\frac{4}{\epsilon}\right)^{pN}e^{-N\lambda^{2}/2}.
\]
Choosing $\lambda$ sufficiently large yields the result. 
\end{proof}
\begin{prop}
\label{prop:A-is-Lipschitz} There exists $c_{2},c>0$ which depend
on $M$ such that 
\[
\pr(\max_{u,v\in[-M,M]^{N}}\frac{\norm{A(\cdot,u)-A(\cdot,v)}_{2}}{\norm{u-v}_{2}}\geq c_{2})\leq\exp(-cN),
\]
for all sufficiently large $N$.
\end{prop}

\begin{proof}
By multilinearity of $A$, the triangle inequality, and the definition of the operator norm,
\begin{align*}
\|A(\cdot,u,\ldots,u)-A(\cdot,v,\ldots,v)\|_2 & \leq\|A(\cdot,u-v,\ldots,u)\|_2+...+\|A(\cdot,v,\ldots,v,u-v)\|_2\\
 & \leq \norm{A}_{\rm op}\max\{\norm{v}_2^{p-2},\norm{u}_2^{p-2}\}\norm{u-v}_2.
 \end{align*}
Since $u,v\in [-M,M]^N$, it follows that $\norm{v}_2,\norm{u}_2\leq M \sqrt{N}$. Applying the bound from 
Lemma~\ref{lemma:op-norm-bound}, we obtain the result.
\end{proof}
\begin{lemma}
\label{lemma:Norm-V-large} There exist $c,C>0$ such that with probability  for every at least $1-\exp(-cN)$ for all large $N$
the following holds: for every
$\eta>0,$ every $u\in H_{N}$ satisfying $A(u)\leq -\eta N$ also satisfies $\norm{u}_{2}\geq C\eta^{1/p}\sqrt{N}$.
\end{lemma}

\begin{proof}
Note that by Lemma \ref{lemma:op-norm-bound}, with probability at least $1-\exp(-cN)$
\[
\abs{\max_{\norm{u}_{2}\leq\delta\sqrt{N}}A(u)}\leq\delta^{p}N^{\frac{p}{2}}\cdot\norm{A}_{\rm op}\leq C\delta^{p}N.
\]
Thus if $A(u)\leq-\eta N$, it must be that 
$\norm{u}_{2}^{p}\norm{A}_{\rm op}\geq\eta N,$
so that 
\[
\norm{u}_{2}\geq\frac{1}{C}\eta^{1/p}N^{1/2},
\]
where $C$ is as in Lemma \ref{lemma:op-norm-bound}. 
\end{proof}

\section{Continuous dependence}\label{section:Iterations-and-bounds}
When we view Algorithm~\ref{alg:AMP} as a discrete time dynamical system, it is natural to expect that 
this admits a similar dependence on the tensor $A$ as a time-inhomogenous
differential equation of the same form. Thus our proof of continuous dependence of iterations on the tensor $A$ 
can be viewed as a discrete analogue of similar  standard result for differential equations, see, e.g.,  \cite[Section 2.4]{teschl2012ordinary}.

Given any tensor $B\in (\R^N)^{\otimes p}$, let 
\begin{align}\label{eq:Max-norm-A(u)}
c_2(B)\triangleq \sup_{u\ne v\in [-M,M]^N}{\|B(\cdot,u)-B(\cdot,v)\|_2\over \|u-v\|_2}.
\end{align}
We now state the main result of this section.
\begin{theorem}\label{theorem:iteration-bounds}
Let $B,\hat{B}\in\left(\R^{N}\right)^{\otimes p}$, and let $\hat{V}^{t},V^{t}$
denote the corresponding sequences output by Step 2 of Algorithm~\ref{alg:AMP},
with the same initial vector $U^{0}=\hat{U}^{0}$. Under 
Assumption~\ref{assumptions:F}, there is some constant $K$ which depends only on $\zeta$ and $c_2(\hat B)$, 
such that for every $T\geq1$ and $U^0$,
\[
\sup_{1\leq t\leq T}\norm{\hat{V}^{t}-V^{t}}\leq K^{T}\norm{\hat{B}-B}_{\rm op}\left(\zeta M\sqrt{N T}\right)^{p-1}.
\]
\end{theorem}

\begin{proof}
Define $U^{t}$ and $\hat{U}^{t}$ as in Step 1 of Algorithm~\ref{alg:AMP},
 and let $\mathbf{U}_{t}=\left(U^{s}\right)_{0\le s\leq t}$ and $\hat{\mathbf{U}}_{t}=\left(\hat{U}^{s}\right)_{0\le s\leq t}$.
Since the map $f(x)=[x]_{M}$ is 1-Lipschitz for any $M$, we see
that the claim of the theorem  follows, provided that
\[
\beta_{N}(t)\triangleq\norm{\mathbf{U}_{t}-\hat{\mathbf{U}_{t}}}_2=\sqrt{\sum_{s\leq t}\norm{U^{s}-\hat{U}^{s}}_{2}^{2}}
\]
satisfies 
\begin{align}\label{eq:beta_bound}
\beta_{N}(t)\leq(K(c_{2}(\hat B)+1))^{t}\norm{\hat{B}-B}_{\rm op}\left(\zeta M\sqrt{Nt}\right)^{p},
\end{align}
as trivially $\|\hat U^t-U^t\|_2\le \beta_{N}(t)$.

Thus we establish (\ref{eq:beta_bound}). By 1-Lipschitz continuity of $[\cdot]_M$ we have
\begin{align*}
&\norm{\hat U^{t+1}-{U}^{t+1}}_2  \leq \\
&\norm{F_{t+1}(\hat{B}(\cdot,f_{t+1}(\hat{U}^{t},\hat{U}^{t-1},\ldots,\hat{U}_{0})),\hat{U}^{0},...,\hat{U}^{t})
-F_{t+1}(B(\cdot,f_{t+1}(U^{0},...,U^{t})),U_{0},\ldots,U^{t})}_{2}.
\end{align*}
Applying  the part of Assumption~\ref{assumptions:F} regarding $F_t$, we see that 
\[
\norm{\hat U^{t+1}-{U}^{t+1}}_2
\leq
\zeta\sqrt{\beta_{N}^{2}(t)+\norm{\hat{B}(\cdot,f_{t+1}(\hat{\mathbf{U}}_{t}))-B(\cdot,f_{t+1}(\mathbf{U}_{t}))}_{2}^{2}},
\]
so that 
\[
\beta_{N}(t+1)\leq (1+\zeta^2)\sqrt{\beta_{N}^{2}(t)+\norm{\hat{B}(\cdot,f_{t+1}(\hat{\mathbf{U}}_{t}))-B(\cdot,f_{t+1}(\mathbf{U}_{t}))}_{2}^{2}}.
\]
By the triangle inequality, 
\begin{align*}
&\norm{\hat{B}(\cdot,f_{t+1}(\hat{\mathbf{U}}_{t}))-B(\cdot,f_{t+1}(\mathbf{U}_{t}))}_{2} \\
&\leq
\norm{\hat{B}(\cdot,f_{t+1}(\hat{\mathbf{U}}_{t}))-\hat{B}(\cdot,f_{t+1}(\mathbf{U}_{t}))}_{2}
+\norm{\hat{B}(\cdot,f_{t+1}(\mathbf{U}_{t}))-B(\cdot,f_{t+1}(\mathbf{U}_{t}))}_2\\
&=I+II.
\end{align*}
By definition of $c_{2}(\hat{B})$,
\[
I\leq c_{2}(\hat{B})\norm{f_{t+1}(\mathbf{\hat{U}}_{t})-f_{t+1}(\mathbf{U}_{t})}_2\leq \zeta c_{2}(\hat{B})\beta_{N}(t).
\]
We now analyze $II$. Note that by Assumption~\ref{assumptions:F}
\begin{align*}
\|f_{t+1}(\mathbf{U}_{t})\|_2^2 &\le \zeta^2\sum_{0\le i\le t}\|U^i\|_2^2 \le \zeta^2(t+1)M^2N.
\end{align*}
Then
\begin{align*}
II\le \|\hat B-B\|_{\rm op} (M\zeta \sqrt{N(t+1)})^{p-1}.
\end{align*}
Combining these bounds, we obtain, 
\[
\norm{B(\cdot,f_{t+1}(\mathbf{U}_{t}))-\hat{B}(\cdot,f_{t+1}(\hat{\mathbf{U}}_{t}))}_{2}\leq 
\zeta c_{2}(\hat B)\beta_{N}(t)+\norm{\hat{B}-B}_{\rm op}\left(\zeta M\sqrt{N(t+1)}\right)^{p-1}.
\]
Plugging this in to the above, yields
\begin{align*}
\beta_{N}(t+1) & \leq 
(1+\zeta^2)\sqrt{\beta_{N}^{2}(t)+\left(\zeta c_{2}(\hat{B})\beta_{N}(t)+\norm{\hat{B}-B}_{\rm op}\left(\zeta M\sqrt{Nt}\right)^{p-1}\right)^{2}}.
\end{align*}
We can write the inequality above in the form
\begin{align*}
\beta_{N}(t+1) \leq K\beta_{N}(t)+b(t),
\end{align*}
where $b(t)$ is non-decreasing, and $K>1$ which depends only on $c_{2}(\hat{B})$ and $\zeta$. The inequality above is a discrete version
of Gronwall's inequality and using $\beta_N(0)=0$, easily leads to a bound

\begin{align*}
\beta_{N}(t) & \leq K^{t}b(t)=K^{t}\norm{\hat{B}-B}_{\rm op}\left(\zeta M\sqrt{Nt}\right)^{p-1}.\qedhere
\end{align*}
\end{proof}

\section{Concentration property of the AMP solution}\label{eq:concentration-of-V}
In this section we establish that the value associated with the solution $V$ produced by the AMP is concentrated around its expectation.

\begin{theorem}\label{theorem:solution-concentrated}
Suppose that Assumption~\ref{assumptions:F} holds. 
For any $\epsilon,M,T,\zeta$, there exists $c>0$ such that 
such that the value $A(V)$ associated with the solution $V$ produced in Step~2 of Algorithm \ref{alg:AMP}  satisfies
\begin{align*}
\sup_{U^0\in[-M,M]^N} \pr\left(|A(V)-\E{A(V)\vert U^0}|\ge \epsilon N\vert U^0\right)\le \exp(-cN),
\end{align*}
for all sufficiently large $N$.
\end{theorem}

\begin{proof}
Fix $U^0$.
Our approach is based on Gaussian concentration combined with Kirszbraun's Theorem. Let $Z\in (\R^N)^{\otimes p}$ denote a tensor
consisting of i.i.d. standard normal entries, so that $A=Z/N^{p-1\over 2}$ in distribution. We let $f(Z)=A(V(A))=Z(V(Z/N^{p-1\over2}))/N^{p-1\over 2}$, 
where $V=V(Z)$ is again the solution produced by AMP viewed as a function of $Z$, and thus introduce $f:(\R^N)^{\otimes p}\to \R$
defined by $f(z)=z(V(z/N^{p-1\over 2}))/N^{p-1\over 2}$. We first establish that this function is Lipschitz with an appropriate constant on an appropriate
subspace of $(\R^N)^{\otimes p}$. Recall the constant $c_2$ introduced in Proposition~\ref{prop:A-is-Lipschitz}. Let 
\begin{align*}
K_{2,N}=\left\{z\in (\R^N)^{\otimes p}: c_2\left({z \over N^{p-1\over 2}}\right)\le c_2\right\}.
\end{align*}
In particular, a random $Z$ with i.i.d. standard normal entries satisfies 
\begin{align}\label{eq:K-2-N}
\pr\left(Z\in K_{2,N}\right)\ge 1-\exp(-cN),
\end{align} for all large enough $N$,
where $c$ is as in the proposition. 

\begin{lemma}\label{lemma:f-lip}
There exists a constant $c=c(M,c_2,\zeta,T)$ such that for every $z_1,z_2\in K_{2,N}$ 
\begin{align*}
|f(z_2)-f(z_1)|\le c\sqrt{N}\|z_2-z_1\|_{2}.
\end{align*}
\end{lemma}

\begin{proof}
Applying Theorem~\ref{theorem:iteration-bounds},   for any $z_1,z_2\in K_{2,N}$ we have
\begin{align}
\|V(z_2/N^{p-1\over 2})-V(z_1/N^{p-1\over 2})\|_2
&\le cN^{p-1\over 2}\|N^{-{p-1\over 2}}(z_2-z_1)\|_{\rm op} \notag\\
&= c\|z_2-z_1\|_{\rm op}, \label{eq:z2-z1-F}
\end{align}
where $c=c(M,c_2,\zeta,T)$ is an appropriate constant.

Next,
\begin{align}
|f(z_1)-f(z_2)| & =N^{-{p-1\over 2}}|z_2(V(z_2/N^{p-1\over 2}))-z_1(V(z_1/N^{p-1\over 2}))| \notag\\
&\le N^{-{p-1\over 2}}|z_2(V(z_2/N^{p-1\over 2}))-z_1(V(z_2/N^{p-1\over 2}))| \notag\\
&+N^{-{p-1\over 2}}|z_1(V(z_2/N^{p-1\over 2}))-z_1(V(z_1/N^{p-1\over 2}))|   \label{eq:fz2-fz1}
\end{align}
We first analyze the second summand above. For simplicity we use  $v_1,v_2$ in place of $V(z_1/N^{p-1\over 2})$ and $V(z_2/N^{p-1\over 2})$.
Note $z_1(u)=\langle u,z(\cdot,u)\rangle$. Thus
\begin{align*}
|z_1(v_2)-z_1(v_1)|&=|\langle v_2,z_1(\cdot,v_2)\rangle-\langle v_1,z_1(\cdot,v_1)\rangle| \\
&\le |\langle v_2,z_1(\cdot,v_2)\rangle-\langle v_2,z_1(\cdot,v_1)\rangle|+|\langle v_2,z_1(\cdot,v_1)\rangle-\langle v_1,z_1(\cdot,v_1)\rangle|.
\end{align*}
Then
\begin{align*}
|\langle v_2,z_1(\cdot,v_2)\rangle-\langle v_2,z_1(\cdot,v_1)\rangle|&=|\langle v_2,z_1(\cdot,v_2)-z_1(\cdot,v_1)\rangle| \\
&\le \|v_2\|_2\|z_1(\cdot,v_2)-z_1(\cdot,v_1)\|_2 \\
&\le M\sqrt{N}N^{p-1\over 2}c_2(z_1/N^{p-1\over 2})\|v_2-v_1\|_2.
\end{align*}
Since $z_1\in K_{2,N}$, we obtain instead a bound 
\begin{align*}
M\sqrt{N}N^{p-1\over 2}c_2\|v_2-v_1\|_2\le M\sqrt{N}N^{p-1\over 2}c_2c\|z_2-z_1\|_{\rm op},
\end{align*}
where the inequality follows from (\ref{eq:z2-z1-F}).

For the second term we have
\begin{align*}
|\langle v_2,z_1(\cdot,v_1)\rangle-\langle v_1,z_1(\cdot,v_1)\rangle| 
&\le 
\|v_2-v_1\|_2\|z_1(\cdot,v_1)\|_2
\end{align*}
Since $z_1\in K_{2,N}$ and $z_1(\cdot,0)=0$, then 
\begin{align*}
\|z_1(\cdot,v_1)\|_2\le N^{p-1\over 2}c_2\|v_1\|_2\le N^{p-1\over 2}c_2M\sqrt{N}.
\end{align*} 
Using (\ref{eq:z2-z1-F}) to $\|v_2-v_1\|_2$, we obtain a bound $c\|z_2-z_1\|_{\rm op}N^{p-1\over 2}c_2M\sqrt{N}$.

Applying both bounds to (\ref{eq:fz2-fz1}) and using $\|\cdot\|_{\rm op}\le \|\cdot\|_2$ we complete the proof.
\end{proof}

We now complete the proof of the theorem.  For every $z\in (\R^N)^{\otimes p}$, define 
\begin{align*}
g(z)=\inf_{\hat z\in K_{2,N}}\left(f(\hat z)+c\sqrt{N}\|\hat z-z\|_2\right),
\end{align*}
where $c$ is as in Lemma~\ref{lemma:f-lip}.
Kirszbraun's Theorem says that $g$ is a Lipschitz continuous function with constant $c\sqrt{N}$ and $g=f$ on $K_{2,N}$. This is easy to verify.
Indeed, fix any $z_1,z_2\in (\R^N)^{\otimes p}$ and $\epsilon>0$. Find $\hat z_1\in K_{2,N}$ such that 
\begin{align*}
|g(z_1)-\left(f(\hat z_1)+c\sqrt{N}\|\hat z_1-z_1\|_2\right)|\le \epsilon.
\end{align*}
Then
\begin{align*}
g(z_2)-g(z_1)&\le f(\hat z_1)+c\sqrt{N}\|\hat z_1-z_2\|_2-\left(f(\hat z_1)+c\sqrt{N}\|\hat z_1-z_1\|_2\right)+\epsilon \\
&=c\sqrt{N}\left(\|\hat z_1-z_2\|_2-\|\hat z_1-z_1\|_2\right)+\epsilon \\
&\le c\sqrt{N}\|z_2-z_1\|_2.
\end{align*}
Using a similar reversed inequality, the Lipschitz continuity of $g$ is established. Now if $z\in K_{2,N}$ then by Lemma~\ref{lemma:f-lip}
for every $\hat z\in K_{2,N}$
\begin{align*}
f(\hat z)+c\sqrt{N}\|\hat z-z\|_2\ge f(z),
\end{align*}
implying that the infimum is achieved by $\hat z=z$, establishing the Kirszbraun's Theorem. 

In conclusion, $g$ is a Lipschitz continuous function with constant $c\sqrt{N}$. Thus by Gaussian concentration (see, e.g., \cite{vershynin2018high}), for every $t\ge 0$
\begin{align*}
\pr\left(|g(Z)-\E{g(Z)\vert U^0}|\ge t N\vert U^0\right)&\le \exp\left(-{t^2N\over 4c^2}\right).
\end{align*}
We now use the fact that $f=g$ on the high probability set $K_{2,N}$. Specifically
\begin{align*}
\E{g(Z)\vert U^0} &= \E{f(Z)\vert U^0}-\E{f(Z)\mathbf{1}\left(Z\notin K_{2,N}\right)\vert U^0}+\E{g(Z)\mathbf{1}\left(Z\notin K_{2,N}\right)\vert U^0}.
\end{align*}
Using $g(Z)\le f(0)+c\sqrt{N}\|Z\|_2=c\sqrt{N}\|Z\|_2$ we have 
\begin{align*}
\E{g(Z)\mathbf{1}\left(Z\notin K_{2,N}\right)\vert U^0} &\le c\sqrt{N}\E{\|Z\|_2\mathbf{1}\left(Z\notin K_{2,N}\right)} \\
&\le c\sqrt{N}\left(\E{\|Z\|_2^2}\right)^{1\over 2}\pr^{1\over 2}\left(Z\notin K_{2,N}\right) \\
&\le \exp(-c_4N),
\end{align*}
for some appropriately chosen $c>0$ and all sufficiently large $N$, where in the second line, we used that $U^0$ and $A$ are independent, and the last inequality follows from (\ref{eq:K-2-N}) and from $\E{\|Z\|_2^2}=N^{O(1)}$.
Similarly, since $f(Z)\le N^{-{p-1\over 2}}\sum |Z_{i_1,\ldots,i_p}|$, we also have $\E{f^2(Z)\vert U^0}=N^{O(1)}$ and thus 
$\E{f(Z)\mathbf{1}\left(Z\notin K_{2,N}\right)\vert U^0}$ is at most $\exp(-c_4N)$ for all large enough $N$, where we used the same notation for constant $c_4$ as above
for convenience. We conclude
\begin{align*}
|\E{g(Z)\vert U^0}-\E{f(Z)\vert U^0}|\le \exp(-c_5N),
\end{align*}
for some $c_5>0$ and all large $N$. 

Thus for any $t>0$
\begin{align*}
\pr\left(|f(Z)-\right.&\left.\E{f(Z)\vert U^0}|\ge t N\right) \\ 
&\le  \pr\left(|g(Z)-\E{f(Z)\vert U^0}|\ge t N, \mathbf{1}\left(Z\in K_{2,N}\right)\right)+\pr(Z\notin K_{2,N}) \\
&\le \pr\left(|g(Z)-\E{g(Z)\vert U^0}|\ge t N-\left(\E{g(Z)\vert U^0}-\E{f(Z)\vert U^0}\right)\right)+\exp(-CN) \\
&\le \pr\left(|g(Z)-\E{g(Z)\vert U^0}|\ge (t/2) N\right)+\exp(-CN), \\
&\le \exp(-c_6N),
\end{align*}
for all large enough $N$ and appropriately chosen $c_6>0$. As $U^0$ was arbitrary
the result then follows. 
%
\end{proof}

\section{OGP is an obstruction to AMP. Proof of the main result}\label{section:OGP-obstructs-AMP}
In this section we complete the proof of the main result, Theorem~\ref{theorem:Main-result}.
Let us begin by first conditioning on the value of $U^0$. 
Let $A\in(\R^N)^{\otimes p}$ be a tensor
with i.i.d. $\mathcal{N}(0,1/N^{p-1})$ entries. Recall that by assumption $A$ and $U^0$ are independent.
Let $V=V(A)$ be the result of the Step 2 of Algorithm \ref{alg:AMP}
after $T$ steps. 
Applying the concentration properties given by Theorems~\ref{theorem:existence-plus-concentration} and~\ref{theorem:solution-concentrated},
it suffices to show  that  for every $\epsilon>0$,
\begin{align*}
{\E{A(V)\vert U^0}\over N}\ge \E{\eta_N}+\mu-\epsilon,
\end{align*}
for all large enough $N$, where $\mu$ is as in Conjecture~\ref{conjecture:OGP-for-tensors-Hilbert-cube}, as in this case
the main result would be established for $\bar\mu=\mu-2\epsilon$ for every $\epsilon>0$.

Thus for the purposes of contradiction, assume
\begin{align}\label{eq:EA(V)}
\E{A(V)\vert U^0}/N\le \E{\eta_N}+\mu_2.
\end{align}
for some $\mu_2<\mu$ for infinitely many $N$.

Generate a tensor $\hat A\in (\R^N)^{\otimes p}$ distributed as $A$ and independent from $A$ and $U^0$.
Consider the interpolated set $\mathcal{A}=(A_\tau, \tau\in [0,1])$  described in Section~\ref{section:OGP}. Denote by $\mathcal{E}_{\rm OGP}$ 
the high probability OGP event defined in Conjecture~\ref{conjecture:OGP-for-tensors-Hilbert-cube} with parameters $\mu,\nu_1,\nu_2$.  
Let $V_\tau$ be the vector
produced by AMP when run on tensor $A_\tau, \tau\in [0,1]$.  For any $\tau\in [0,1]$ we have
\begin{align*}
c_2(A_\tau)&=
\sup_{u\ne v\in [-M,M]^N}{\|\sqrt{1-\tau}A(\cdot,u)+\sqrt{\tau}\hat A(\cdot,u)-(\sqrt{1-\tau}A(\cdot,v)+\sqrt{\tau}\hat A(\cdot,v))\|_2\over \|u-v\|_2} \\
&\le \sqrt{1-\tau}c_2(A)+\sqrt{\tau}c_2(\hat A)\\
&\le c_2(A)+c_2(\hat A).
\end{align*}
Applying Proposition~\ref{prop:A-is-Lipschitz} we have $c_2(A)+c_2(\hat A)\le 2c_2$ modulo exponentially small in $N$ probability.
We conclude that $\sup_{\tau\in [0,1]}c_2(A_\tau)\le 2c_2$ modulo exponentially small probability.

By Theorem~\ref{theorem:iteration-bounds} then modulo exponentially small probability, using $\|\cdot\|_{\rm op}\le \|\cdot\|_2$, 
we have that for any $\tau_1,\tau_2\in [0,1]$
\begin{align*}
\|V^{\tau_1}-V^{\tau_2}\|_2\le C^TN^{p-1\over 2}\|A_{\tau_1}-A_{\tau_2}\|_2,
\end{align*}
for some constant $C>0$ which incorporates $c_2,\zeta$, and $M$ (and which may change from line to line). We have
\begin{align*}
\|A_{\tau_1}-A_{\tau_2}\|_2 &= \|\sqrt{1-\tau_1}A-\sqrt{1-\tau_2}A+\sqrt{\tau_1}\hat A-\sqrt{\tau_2}\hat A \|_2 \\
&\le \left(|\sqrt{1-\tau_1}-\sqrt{1-\tau_2}|+|\sqrt{\tau_1}-\sqrt{\tau_2}|\right)\left(\|A\|_2+\|\hat A\|_2\right).
\end{align*}
Since $\|A\|_2^2$ is distributed as $N^{-{p-1}}\sum_{1\le i\le N^p}Z_i^2$, which is $N$ in expectation, then by standard large deviations estimates, 
$\pr(\|A\|_2\ge cN)$ is exponentially small for any $c>1$. In particular, $\|A\|_2+\|\hat A\|_2\le 4\sqrt{N}$, modulo exponentially small probability.

Combining and letting $h(\tau_1,\tau_2)=|\sqrt{1-\tau_1}-\sqrt{1-\tau_2}|+|\sqrt{\tau_1}-\sqrt{\tau_2}|$ we obtain
\begin{align}\label{eq:uniform-tau_1-tau_2}
|\|V^{\tau_1}\|_2-\|V^{\tau_2}\|_2| \le
\|V^{\tau_1}-V^{\tau_2}\|_2\le C^TN^{p\over 2}h(\tau_1,\tau_2),
\end{align}
for all $\tau_1,\tau_2$ modulo exponentially small in $N$ probability.

Next, for any $\tau_1,\tau_2\in [0,1]$
\begin{align*}
{\langle V^0,V^{\tau_2}\rangle \over \|V^0\|_2\|V^{\tau_2}\|_2}-
{\langle V^0,V^{\tau_1}\rangle \over \|V^0\|_2\|V^{\tau_1}\|_2}
&=
{\|V^{\tau_1}\|_2\langle V^0,V^{\tau_2}\rangle- \|V^{\tau_2}\|_2\langle V^0,V^{\tau_1}\rangle \over \|V^0\|_2\|V^{\tau_1}\|_2\|V^{\tau_2}\|_2}.
\end{align*}
For the numerator, applying the Cauchy-Scwhartz inequality
\begin{align*}
&|\|V^{\tau_2}\|_2\langle V^0,V^{\tau_1}\rangle- \|V^{\tau_1}\|_2\langle V^0,V^{\tau_2}\rangle| \\
&=
|\|V^{\tau_2}\|_2\langle V^0,V^{\tau_1}\rangle-\|V^{\tau_1}\|_2\langle V^0,V^{\tau_1}\rangle
+\|V^{\tau_1}\|_2\langle V^0,V^{\tau_1}\rangle- \|V^{\tau_1}\|_2\langle V^0,V^{\tau_2}\rangle| \\
&\le |\|V^{\tau_2}\|_2-\|V^{\tau_1}\|_2 | ~\|V^0\|_2\|V^{\tau_1}\|_2+
\|V^{\tau_1}\|_2 \|V^0\|_2\|V^{\tau_2}-V^{\tau_1}\|_2
\end{align*}
Provided (\ref{eq:uniform-tau_1-tau_2}) holds, we obtain
\begin{align*}
\Big| {\langle V^0,V^{\tau_2}\rangle \over \|V^0\|_2\|V^{\tau_2}\|_2}-
{\langle V^0,V^{\tau_1}\rangle \over \|V^0\|_2\|V^{\tau_1}\|_2} \Big| \le
{2C^TN^{p\over 2}h(\tau_1,\tau_2)\over \|V^{\tau_2}\|_2}.
\end{align*}
Next we fix $\alpha>0$, to be specified later, let $\delta=N^{-\alpha}$. We assume for convenience that $1/\delta=N^\alpha$ is an integer.
Introduce  a discrete sequence $\tau_n=n\delta, 0\le n\le 1/\delta$. By Lemma~\ref{lemma:Norm-V-large}, applying union in $N^\alpha$
terms bound, we have that $\|V^{\tau_n}\|\ge C_2\sqrt{N}$ for some $C_2>0$ for all sufficiently large $N$, modulo exponentially small probability. 
Provided this holds, the bound above can be replaced by 
\begin{align*}
{2C^TN^{p\over 2}h(\tau_1,\tau_2)\over C_2\sqrt{N}}=C^Th(\tau_1,\tau_2)N^{p-1\over 2}.
\end{align*}
Now we have
\begin{align*}
h(\tau_{n_1},\tau_{n_2})&\le C^TN^{\alpha\over 2}(\tau_{n_2}-\tau_{n_1}) =C^TN^{\alpha\over 2}N^{-\alpha}(n_2-n_1)
=C^T N^{-{\alpha\over 2}}(n_2-n_1).
\end{align*}
for all $n_1,n_2$ and some $C>0$.

Combining, we conclude that modulo exponentially small in $N$ probability, for all $n=0,\ldots,N^\alpha$,
\begin{align*}
\Big |
{\langle V^0,V^{\tau_{n+1}}\rangle \over \|V^0\|_2\|V^{\tau_{n+1}}\|_2}-
{\langle V^0,V^{\tau_n}\rangle \over \|V^0\|_2\|V^{\tau_n}\|_2}
\Big| \le C^T N^{p-1\over 2}N^{-{\alpha\over 2}}
\end{align*}
and provided $\alpha>p-1$ the bound above is $o(1)$, and in particular is smaller than $\nu_2-\nu_1$ for $N$ sufficiently large.

Next we examine ${\langle V^0,V^{\tau_n}\rangle \over \|V^0\|_2\|V^{\tau_n}\|_2}$ in the extreme case $n=0$ and $n=N^\alpha$.
The value is clearly $1$ when $n=0$. Applying the second part of Conjecture~\ref{conjecture:OGP-for-tensors-Hilbert-cube} we have
that for every $\epsilon$, this value is at most $\epsilon/C^2$ modulo exponentially small probability, where $C$ is the constant from
Lemma~\ref{lemma:Norm-V-large}. In particular at $n=N^\alpha$ this value is at most $\nu_1$. It follows that there must exist 
an index $n^*$, (which is random in general) such that 
\begin{align*}
\Big|{\langle V^0,V^{\tau_{n^*}}\rangle \over \|V^0\|_2\|V^{\tau_{n^*}}\|_2}\Big|\in (\nu_1,\nu_2).
\end{align*}
Now in the event that OGP holds, which by Conjecture~\ref{conjecture:OGP-for-tensors-Hilbert-cube} holds modulo exponentially small probability,
this implies $A(V^{\tau_{n^*}})/N\ge \E{\eta_N}+\mu$ and therefore the larger event
\begin{align*}
\max_{0\le n\le N^\alpha}A(V^{\tau_{n}})/N\ge \E{\eta_N}+\mu.
\end{align*}
However, this contradicts assumption (\ref{eq:EA(V)}) and the concentration bound of Theorem~\ref{theorem:solution-concentrated} applied
in the union over $n=0,\ldots,N^\alpha$ bound. This yields the result conditionally on $U^0$. 
Since the lower bound we obtain does not depend on $U^0$, we can take the expectation in $U^0$,
and obtain the  main result. 

%

\section{TAP-type iteration schemes}\label{section:TAP-iteration}

One motivations for the AMP algorithm discussed in the introduction
is the prediction that the minimizers of $A(u)$ satisfy a self-consistent equation, 
called a ``mean-field'' equation. In this setting, equations of this type are called Thouless--Anderson--Palmer (TAP) equations, 
after the work of those three authors in~\cite{thouless1977solution}, on mean-field equations in the case $p=2$ on $\calB_N$, 
in a certain physically motivated relaxation. For a discussion of these and related results see also \cite{MezardParisiVirasoro}.
In this section, we show  that, as an implication of Theorem~\ref{theorem:Main-result}, 
the iterative methods designed to  produce solutions to TAP-like equations  fail, modulo Conjecture\ref{conjecture:OGP-for-tensors-Hilbert-cube}.

More precisely, consider the following modification of the objective. 
Recall the Bernoulli entropy, $S:[-1,1]\to\R_+$ 
\[
S(x) = \frac{1}{2}(1+x)\log(1+x)+\frac{1}{2}(1-x)\log(1-x) .
\]
For any $\beta>0$, let $f_\beta:[-1,1]\to \R$
\[
f_\beta(x)= \frac{\beta^2}{2}(1-x^p-p x^{p-1}(1-x)),
\]
Finally, define the one-parameter family of functions $F_\beta:\calH_N\to\R^N$ given by 
\[
F_\beta(x) = \beta A(x) - S(x) +f_\beta(\frac{\norm{x}^2}{N}).
\]
Observe that as $\beta\to\infty$,
\[
\frac{F_\beta(u)}{\beta}
\to \begin{cases}
A(x) & x\in \calB_N,\\
\infty & x\in \calH_N\setminus \calB_N.
\end{cases}
\]
Thus one expects that for $\beta$ very large, minimizers of $F_\beta$
are near minimizers for $A(u)$. In particular, one approach to computing near minimizers
for \eqref{eq:GroundState} would be to produce minimizers of $F$.

Differentiating $F_\beta$, we see that the critical points of $F$ satisfy the fixed point equation
\begin{equation}\label{eq:tap}
x= \tanh\left[\beta \nabla A(x)  +2f'\left(\frac{\norm{x}^2}{N}\right) x\right].
\end{equation}
Thus one approach to produce these minimizers is to construct solutions of these fixed point equations.
The AMP algorithm is one such method, based off of a deep intuition in the physics literature that suggests
that in the case $p=2$. 

Another approach, would be a more naive approach, in the spirit of standard AMP iterations
would be to simply iteratively construct solutions to \eqref{eq:tap} as in \cite{bolthausen2014iterative}.
It is expected that the critical points of this equation satisfy,
\[
\frac{\norm{x}^2}{N} = q_{*}(\beta),
\]
where $q_{*}(\beta)$ is an explicit constant, called the Edwards-Anderson order parameter
and is given by as in \cite{montanari2018optimization}. 
Motivated by this, consider the following class of AMP iterations, 
\[
U^t = \tanh\left(\beta A(\cdot, U^{t-1})  + a_{t-2}U^{t-2}\right),\quad U_0 = 1_N q, ~U_{-1} = 0\quad 1\le t\le T.
\]
where $q>0$ is a fixed constant and $a_t$ is any bounded sequence. For instance, we may take $a_t= 2f'(q_*)$ and $q=q_*$. One 
might make the replacement $a_t\mapsto f'(\norm{U^{t}}_2^2/N)$, however, one can show that this will not change the performance 
if the original sequence was chosen appropriately. See Section~\ref{section:AMP-p=2} for a similar argument in the more detailed 
case of the AMP iteration from \cite{montanari2018optimization}.   

As a consequence of OGP, we see that the above iteration will fail to produced fixed points which are also near optimizers of 
$A(u)$ for $\beta$ large. More precisely we have the following.
\begin{coro}
Suppose that the entries of $A\in (\R^N)^{\otimes p}$ are i.i.d. $\N(0,N^{-p+1})$ and $p\geq 4$ is even.
Assuming the validity of Conjecture~\ref{conjecture:OGP-for-tensors-Hilbert-cube}, there exists a $\bar\mu>0$
such that for any $M,T>0$ and $\beta$ sufficiently large, if $V=V(A)$ is the result of the firs two steps of the AMP algorithm after $T$ iterations, 
then  
\[
\frac{1}{\beta}{ F_\beta(V)} \geq \frac{\min_{x\in \calH_N} F(x)}{N} +\bar{\mu}.
\]
\end{coro}
\begin{proof}
First note that this iteration is of the form \eqref{eq:Iteration-U^t},
for some functions $F_t,f_t$ satisfying Assumption~\ref{assumptions:F}.
Indeed, let 
\begin{align*}
f_t(u_1,\ldots,u_t) &=u_t\\
F_t(u_0,\ldots,u_t) &= \tanh( \beta u_0+ a_{t-2}  u_{t-2}).
\end{align*}
These functions are Lipschitz on the relevant domains as $\tanh(x)$ is smooth with bounded derivatives. 
Thus by Theorem~\ref{theorem:Main-result}, 
\[A(V) \geq \min_{x\in \calB_N} A(x)/N+\bar\mu.
\]

Now observe that $F_\beta(x)$ satisfies $F_\beta(x) \geq \beta A(x)-\log(2)$ on $\calH_N$. In particular,
this is an equality on $\calB_N$.
As a result,
\begin{align*}
\frac{1}{N\beta}{F_\beta(V)} &\geq \frac{A(V)}{N} -\frac{\log 2}{\beta N}
\geq \min_{x\in \calB_N} \frac{A(x)}{N}+\bar\mu \\
&\geq \min_{x\in\calB_N}\frac{F_\beta(x)}{\beta N} +\bar\mu + \frac{\log(2)}{2\beta}
\geq \min_{x\in \calH_N}\frac{F_\beta(x)}{\beta N} +\bar\mu/2,
\end{align*}
where in the last line we take $\beta>\frac{\log(2)}{2\bar\mu}$ by assumption.
\end{proof}

\section{Verification for AMP for $p$-spin models}\label{section:AMP-p=2}
In this section we show that the AMP algorithm defined in~\cite{montanari2018optimization} is a special case of the AMP defined in 
Section~\ref{section:AMP-formalist}, modulo some truncation and averaging steps which we discuss below. Here $p=2$
so $A\in \R^{N\times N}$ is a matrix.
The algorithm constructed in~\cite{montanari2018optimization} is as follows. 
A one-dimensional measure $\mu$ is fixed which is a solution of the minimization problem
of the Parisi functional. A function $\Phi:[0,1]\times \R\to\R$ is a solution of the associated PDF. It is known that $\partial_{x}\Phi(t,x)$
and $\partial_{xx}\Phi(t,x)$ of this function are Lipschitz continuous.
A certain value $q_*\in [0,1]$ is fixed
(it is Edward-Anderson parameter). Let $u^{-1}=0\in\R^N, u^0$ be i.i.d. standard normal vector in $\R^N$: $u^0\stackrel{d}{=} \mathcal{N}(0,I_N)$,
$g^{-2}=0\in\R^N, g^{-1}=1_N\in\R^N, b^0=0\in \R^N$. Given $a,b\in \R^N$, $a\cdot b\in \R^N$ denotes a coordinate-wise product 
of $a$ and $b$. Then for $t=0,1,\ldots,\lfloor q_*/\delta\rfloor\triangleq T$,
\begin{align}
u^{t+1} &= A(g^{t-1}\cdot u^t)-b^tg^{t-2}\cdot u^{t-1}, \label{eq:u_t}\\
x^t &= x^{t-1}+\beta^2\mu(t\delta)\partial_x\Phi(t\delta,x^{t-1})\delta+\beta\sqrt{\delta}u^t, \label{eq:x_t}\\
g^t &= \sqrt{N}\partial_{xx}\Phi(t\delta,x^t)/\|\partial_{xx}\Phi(t\delta,x^t)\|_2 , \label{eq:g_t}\\
b^t &= N^{-1}\sum_{1\le i\le N}g^t_i, \label{eq:b_t}
\end{align}
where  everywhere the functions are applied coordinate-wise.

We first consider modifications of these iterations and justify them. First set $M$ to be a large constant. 
Since $u^0\stackrel{d}{=} \mathcal{N}(0,I_N)$,
then the fraction of coordinates of $u^0$ with absolute values larger than $N$ decreases to zero as a function of $M$. 
Replace (\ref{eq:u_t}) by 
\begin{align}\label{eq:u_t-truncation}
u^{t+1} &= [A(g^{t-1}\cdot u^t)-b_tg^{t-2}\cdot u^{t-1}]_{M}.
\end{align}
In the final step of algorithm in~\cite{montanari2018optimization}, the resulting vectors $u^1,\ldots,u^t$ are used to construct
\begin{align*}
z=\sqrt{\delta}\sum_{1\le t\le \lfloor q_*/\delta\rfloor}g^{t-1}\cdot u^t,
\end{align*}
and then $z$ is rounded to a  vector in $\calH_N$ via $\max(-1,\min(1,\cdot))$ operator. 
It is thus expected that the truncation of $u^t$ in (\ref{eq:u_t-truncation})
by a value $M$  does not affect the result significantly provided $M$ is large, 
though we do not prove this fact.

Next, as an implication of the analysis in~\cite{montanari2018optimization}, as $N\to\infty$, the norm $\|\Phi(t\delta,x^t)\|_2$ is concentrated around
a deterministic function of $t$, which which has value $\Theta(\sqrt{N})$. In particular,  there it is argued that 
the empirical measure $\mathcal{E}^t_N = \frac{1}{N}\sum\delta_{x^{t}_i}$ converges to some deterministic limit $\mathcal{E}^t$ weakly almost surely.
Since $\partial_{xx}\Phi$  is smooth and bounded \cite[Theorem 4]{JagTobSC15}, it then follows that 
\[
\norm{\partial_{xx} \Phi(t\delta,x^t)}/\sqrt{N}=\sqrt{\int \partial_{xx} \Phi(t\delta,y)d\calE_N^t(y)}\to h(t).
\]
Denoting this function by $\sqrt{N}h(t), t=0,1,\ldots, T$, 
we thus rewrite (\ref{eq:g_t}) as 
\begin{align}
g^t=h^{-1}(t)\partial_{xx}\Phi(t\delta,x^t) \label{eq:g_t-2}.
\end{align}
Similarly, $b^t$, which per (\ref{eq:b_t}) is defined as coordinate-wise  average of $g^t$, as $N\to\infty$
is concentrated around a deterministic function of $t$, which we denote by $\eta(t), t=0,1,\ldots,T$. Thus we replace (\ref{eq:b_t}) by
\begin{align*}
b^t=\eta^t.
\end{align*}

We now fit these iterations into our framework defined in Section~\ref{section:AMP-formalist}. We begin by defining $f_t$.
In light of (\ref{eq:g_t-2}) replacing (\ref{eq:g_t}) we may define $f_t:\R^{t+1}\to\R$ as follows
\begin{align}
f_t(u^0,\ldots,u^t)=g^{t-1}(u^0,\ldots,u^{t-1})u^t, \label{eq:f_t}
\end{align}
where the function $g_t:\R^t\to\R$ is a one-dimensional version of $g^t$, namely
\begin{align*}
g_t(u^0,\ldots,u^t)=h^{-1}(t)\partial_{xx}\Phi(t\delta,x^t),
\end{align*}
where $x^t$ is defined through a one-dimensional version of (\ref{eq:x_t}):
\begin{align*}
x^t=x^{t-1}+\beta^2\mu(t\delta)\partial_x\Phi(t\delta,x^{t-1})\delta+\beta\sqrt{\delta}u^t.
\end{align*}
Since $\partial_{xx}\Phi$ and $\partial_{xx}\Phi$ are Lipschitz continuous in the second argument for each fixed first argument, 
it is then immediate  to verify that $g_t:\R^t\to\R$ is Lipschitz continuous as a function of $u^0,\ldots,u^t$, wrt $\|\cdot\|_\infty$ norm,
say with constant $C_t$ and satisfies $g_t(0)=0$. This implies $\|g_t(x)\|_\infty \le C_t\|x\|_\infty$.
Then by (\ref{eq:f_t}) we have 
for every $u^0,\ldots,u^t$ and $v^0,\ldots,v^t$
\begin{align*}
|f_t(u^0,\ldots,u^t)-f_t(v^0,\ldots,v^t)| &= |g^{t-1}(u^0,\ldots,u^{t-1})u_t-g^{t-1}(v^0,\ldots,v^{t-1})v^t| \\
&\le C_t\max_j(\max(|u^j|,|v^j|)\|(u^0,\ldots,u^t)-(v^0,\ldots,v^t)\|_\infty.
\end{align*}
Thus $f_t$ is Lipschitz continuous on $[-M,M]^{t+1}$, and thus satisfies the first part of Assumption~\ref{assumptions:F}.

Now motivated by (\ref{eq:u_t}) or rather (\ref{eq:u_t-truncation}) we  define $F_t:\R^{t+1}\to\R$ by
\begin{align*}
F_t(y,u^0,\ldots,u^t)&=y-b_tg^{t-2}\cdot u^{t-1}  \\
&= y-\eta(t)h^{-1}(t-2)\partial_{xx}\Phi((t-2)\delta,x_t)\cdot u^{t-1}
\end{align*}
We see that $F_t$ is Lipschitz continuous on $\R\times [-M,M]^{t+1}$ and thus satisfies the second part of Assumption~\ref{assumptions:F}.

\section{Some open questions}\label{section:Conclusions}
We now list some questions which remain open. First validating Conjectures~\ref{conjecture:OGP-for-tensors-Hilbert-cube} 
and~\ref{conjecture:near-binary} is of interest. It is conceivable that the first of these conjectures can be approached by analyzing the Parisi measure
directly on the Hilbert cube $\calH_N$, as opposed to the binary cube $\calB_N$. Carrying out the corresponding technical analysis 
of the associated variational problem could be quite daunting though. 
Another interesting question left open in this work is establishing the negative result for the binary
output $\Pi(V)$ of the AMP scheme (Step 3) directly, as opposed to one for the pen-ultimate state $V$. Lifting the truncation $[\cdot]_M$
assumption adopted by our class of AMP algorithms is another question which remains open.

Next, it would be interesting to extend the main result of this paper to the 
$p$-spin spherical spin glass model and complement the positive result of Subag~\cite{subag2018following}. 
We expect that our negative result extends to this model within the same scope of algorithms almost verbatim. 
Furthermore,  by similar arguments one can show that (continuous time) gradient flow started from a uniform 
at random point fails to reach near minimizers of the spherical $p$-spin model in $\log(N)$ time with probability
tending to 1.
That being said, it is important to note
that Subag's algorithm is based on an iterative sequence of computations which involve linear projections of gradient and Hessians
of $A(u)$ to the linear space orthogonal to $u$. As such this computational scheme
does not formally fit our framework of algorithms. It is conceivable  though that the projection step can be approximated well by iterations
of the form we consider, say perhaps by imitating the power iteration approach for spectral computations. Perhaps as an easier challenge,
one could try to show that Subag's scheme specifically fails to find
near ground states in models exhibiting OGP. A related question is whether there exists a connection between the OGP and the algorithmic
hardness of the REM model discussed in~\cite{addario2018algorithmic}.

Our approach was formulated in terms of bounded ($N$-independent) number of iterations $T$. It is easy to see though that the proof method extends
without a change to the case $T\le c\log N$ for small enough constant $c$. At the same time we believe that AMP scheme is not effective in computing
near ground states, regardless of the scale of the number of iterations. Thus an interesting open question is to see whether an AMP scheme 
achieving near ground states can be designed say when $T=N^{O(1)}$. 

Finally, perhaps the most intriguing question which remains open is one regarding the genuine hardness of the problem of finding ground states in models
exhibiting the OGP. While formal hardness of problems associated with spin glass models is known, in particular it is shown in~\cite{gamarnik2018computing} 
that computing the partition function of the $p$-spin models is hard on average even in $p=2$ regime, these results are established using more ''standard''
average case hardness proof approaches, and do not take advantage of the intricate solution space topology, such as the one expressed by OGP. 
At the same time, as of now we have very compelling consistence of the presence of OGP 
and the apparent hardness of the associated optimization problem in many models. What is lacking, however, is the
formal link between the two within a class of algorithms which is broader than AMP. An interesting and challenging conjecture is that the OGP
implies formal average case hardness of the underlying optimization problem, perhaps even within the class of all polynomial time algorithms.

\section*{Acknowledgements} The first author  acknowledges the support from the Office of Naval Research Grant N00014-17-1-2790.
The second author acknowledges the partial support of the National Science Foundation Grant NSF OISE-1604232.

\bibliographystyle{amsalpha}

\newcommand{\etalchar}[1]{$^{#1}$}
\providecommand{\bysame}{\leavevmode\hbox to3em{\hrulefill}\thinspace}
\providecommand{\MR}{\relax\ifhmode\unskip\space\fi MR }
\providecommand{\MRhref}[2]{%
  \href{http://www.ams.org/mathscinet-getitem?mr=#1}{#2}
}
\providecommand{\href}[2]{#2}

\end{document}